\documentclass{amsart}

\usepackage{hyperref}
\hypersetup{
    colorlinks=true,
    linkcolor=blue,
    citecolor=magenta,      
    urlcolor=cyan,
    }

\usepackage{amssymb}
\usepackage{tikz}
\usetikzlibrary{automata, positioning, arrows}
\usepackage{todonotes}
\usepackage{listings}
\theoremstyle{plain}
\newtheorem{theorem}{Theorem}[section]
\newtheorem{corollary}[theorem]{Corollary}
\newtheorem{lemma}[theorem]{Lemma}
\newtheorem{proposition}[theorem]{Proposition}

\theoremstyle{definition}

\newtheorem{example}[theorem]{Example}

\theoremstyle{remark}
\newtheorem{remark}[theorem]{Remark}
\newtheorem{question}[theorem]{Question}

\newmuskip\pFqmuskip
\newcommand*\pFq[6][8]{%
  \begingroup 
  \pFqmuskip=#1mu\relax
  \mathchardef\normalcomma=\mathcode`,
  \mathcode`\,=\string"8000
  \begingroup\lccode`\~=`\,
  \lowercase{\endgroup\let~}\pFqcomma
  {}_{#2}F_{#3}{\left[\genfrac..{0pt}{}{#4}{#5};#6\right]}%
  \endgroup
}
\newcommand{\pFqcomma}{{\normalcomma}\mskip\pFqmuskip}

\newcommand{\Z}{\mathbb{Z}}
\newcommand{\calS}{\mathcal{S}}
\newcommand{\bA}{\mathbf{A}}
\newcommand{\bB}{\mathbf{B}}
\newcommand{\bC}{\mathbf{C}}
\newcommand{\bD}{\mathbf{D}}
\newcommand{\bE}{\mathbf{E}}
\newcommand{\bF}{\mathbf{F}}
\newcommand{\bX}{\mathbf{X}}
\newcommand{\bY}{\mathbf{Y}}
\newcommand{\bL}{\mathbf{L}}

\begin{document}

\title[Walks ending on a coordinate hyperplane]{Lattice walks ending on a coordinate hyperplane avoiding backtracking and repeats}


\author{John Machacek}




\begin{abstract}
We work with lattice walks in $\mathbb{Z}^{r+1}$ using step set $\{\pm 1\}^{r+1}$ that finish with $x_{r+1} = 0$. We further impose conditions of avoiding backtracking (i.e. $[v,-v]$) and avoiding consecutive steps (i.e. $[v,v]$) each possibly combined with restricting to the half-space $x_{r+1} \geq 0$. We find in all cases the generating functions for such walks are algebraic and give explicit formulas for them. We also find polynomial recurrences for their coefficients. 
From the generating functions we find the asymptotic enumeration of each family of walks considered.
The enumeration in special cases includes central binomial coefficients and Catalan numbers as well as relations to enumeration of another family of walks previously studied for which we provide a bijection.
\end{abstract}
\maketitle

\section{Introduction}
\label{sec:intro}
Central binomial coefficients and Catalan numbers are two fundamental sequences in enumerative combinatorics as well as in other areas of mathematics.
Both these numbers can be thought of as enumerating walks in $\mathbb{Z}$ with steps from $\{\pm 1\}$ that end at the origin.
The Catalan number model has the additional restriction the walks must stay in the nonnegative integers and correspond to the height of usual Dyck paths.
We generalize this situation by looking at walks in $\mathbb{Z}^{r+1}$ with steps from $\{\pm 1\}^{r+1}$.
The condition that the walks must end on the hyperplane defined by $x_{r+1} = 0$ is imposed.
For the generalized Catalan situation we also require the walks be confined to the halfspace with $x_{r+1} \geq 0$.
We then look at such walks avoiding the backtracking pattern $[v,-v]$ and the repeat pattern $[v,v]$ for $v \in \{\pm 1\}^{r+1}$.
Here a walk is a sequence of vectors from $\{\pm 1\}^{r+1}$ and $[v_1, v_1, \dots, v_{\ell}]$ contains the pattern $[u_1, u_2]$ if and only if $[v_i, v_{i+1}] = [u_1, u_2]$ for $1 \leq i < \ell$.
That is, our pattern avoidance means that no contiguous subsequence of a walk is equal to the pattern.
When $r=0$ there are no nonempty walks ending at the origin avoiding the backtracking pattern.
While for avoiding the repeat pattern there are the two walks of length $2n$ namely $[+1, -1]^n$ and $[-1,+1]^n$ only the first of which is relevant the Catalan situation.
However, for $r > 0$ there are many such walks.
Our aim is enumerating these walks.

Walks in integer lattices are a topic of interest in computer science, mathematics, physics, and statistics.
Lattice path combinatorics can be approached from various perspectives (see~\cite{Kratt} for survey with focus on enumeration).
Our focus is enumeration (i.e. finding recurrences, generating functions, and formulas).
One such enumeration problem involving avoidance of backtracking we solve gives an interpretation of sequence A085363 in the OEIS~\cite{OEIS} in terms of 2D lattices walks.
This sequence was an original motivation, but we found the techniques readily generalized to arbitrary dimension as well as avoidance of consecutive steps.
In addition to combinatorial arguments, we make use of formal language theory and its connection to generating functions.
We also use automated methods for dealing with hypergeometric terms and finding polynomial recurrences (see~\cite{AB} for more about these types of techniques).

Lattice path enumeration has a long and rich history, and recent attention has been given to avoiding patterns as we consider here.
We saw for Dyck paths (i.e. $r=0$) that avoiding consecutive steps or backtracking is not interesting.
However, the consecutive step patterns $[+1,+1]$ and $[-1,-1]$ are special cases of runs $[+1,+1, \dots, +1]$ and $[-1,-1,\dots,-1]$.
Study of Dyck paths avoiding runs of given lengths (as well as peaks and valleys avoiding certain heights) is conducted in~\cite{EZ} with automated computational methods and in~\cite{BD} with formal grammar techniques.
Avoiding a pattern in lattice paths was considered in~\cite{MR4058415} and later for multiple patterns in~\cite{AoA,MR4138722}.
These latter works make use of the so called vectorial kernel method which modifies the kernel method to allow for pattern avoidance.
While the vectorial kernel method is an efficient and flexible way to access the generating functions of our walks, we want to focus in this work on binomial expressions for the coefficients; so, for our set of jumps we present a way to do so via an alternative approach based on context-free grammars and ad-hoc recurrences.

The remainder of the paper is organized as follows.
In Section~\ref{sec:setting} we formally define the languages of walks we are interested in and show that their generating functions are algebraic (and hence the sequences of coefficients satisfy a recurrence with polynomial coefficients).
Formulas and recurrences are found in Section~\ref{sec:formula}, then generating functions and asymptotics are given in Section~\ref{sec:gen}.
Lastly, Section~\ref{sec:conclusion} concludes by looking at some open problems and related work.
Also, there is the~Appendix which provides Maple code used to prove some results from earlier in the article.

\subsection*{Acknowledgments} The author wishes to thank anonymous referees for their careful reading as well as for several thoughtful and helpful comments which have improved this paper.

\section{The setting}
\label{sec:setting}
For $r \geq 0$ we consider lattice walks in $\Z^{r+1}$ that use steps from $\calS_r := \{\pm 1\}^{r+1}$ which start at the origin and end with $x_{r+1} = 0$.
Let $\bA^{(r)}$ be the language over the alphabet $\calS_r$ consisting of such walks.
Evidently any such walk must consist of $2n$ steps for some $n \geq 0$.
Let $a^{(r)}_n$ be the number of such walks with $2n$ steps.

We are further interested in subsets of these walks with additional conditions.
We consider the language of walks $\bB^{(r)} \subseteq \bA^{(r)}$ which avoid backtracking.
That is, walks ending with $x_{r+1}=0$ with the additional constraint that a step $v$ cannot be directly followed by $-v$ for any $v \in \calS_d$.
We let $b^{(r)}_n$ be the number of walks avoiding backtracking with $2n$ steps.
Similarly, we let $\bC^{(r)} \subseteq \bA^{(r)}$ denote the language of walks avoiding consecutive repeats. 
That is,  walks avoiding a step $v$ directly followed by $v$ for any $v \in \calS_d$.
We let $c^{(r)}_n$ denote the number of walks avoiding consecutive repeats with $2n$ steps.

Now let $\bD^{(r)}$ be the language consisting of walks starting at the origin for which $x_{r+1} \geq 0$ at all times and finish with $x_{r+1} = 0$.
We similarly let $\bE^{(r)} \subseteq \bD^{(r)}$ and $\bF^{(r)} \subseteq \bD^{(r)}$ denote the sublanguages which avoid backtracking and consecutive equal steps respectively.
Again we let the sequences $d^{(r)}_n$, $e^{(r)}_n$, and $f^{(r)}_n$ count the number of walks of length $2n$ in each of the languages 
$\bD^{(r)}$, $\bE^{(r)}$, and $\bF^{(r)}$.
Letting $\bX^{(r)}$ denote the language of walks avoiding the backtracking pattern $[v,-v]$ and $\bY^{(r)}$ denote the language of walks avoiding consecutive repeats $[v,v]$ we have
\begin{align*}
    \bB^{(r)} &= \bA^{(r)} \cap \bX^{(r)} & \bE^{(r)} &= \bD^{(r)} \cap \bX^{(r)} \\
    \bC^{(r)} &= \bA^{(r)} \cap \bY^{(r)} & \bF^{(r)} &= \bD^{(r)} \cap \bY^{(r)} \\
\end{align*}
as expressions of our languages of interest.
It is worth noting that both $\bX^{(r)}$ and $\bY^{(r)}$ are regular languages.

\begin{remark}
Here are a few remarks about notation (some of which has already been used in the Section~\ref{sec:intro}).
We will use $(\;\cdot\;)$ to denote vectors (i.e. elements of our alphabet) and $[\;\cdot\;]$ to denote words (i.e. sequences of vectors).
For example, $(+1,-1,+1) \in \calS_2$ and $[(+1,-1,+1), (+1, +1, -1)]$ is a word over $\calS_2$.
We may also use $(v,+1)$ or $(v,-1)$ to denote a element of $\calS_{r+1}$ ending in $+1$ or $-1$ respectively for $v \in \calS_r.$
We will also use the power notation to denote repeated instances of a vector or word.
So if $v = (+1,-1)$ then $[v^3] = [(+1,-1), (+1,-1),(+1,-1)]$.
We may also put words to a power.
For example, if $w_1 = [(+1,+1), (-1,-1)]$ and $w_2 = [(+1,+1),(+1,+1)]$, then \[[w_1^2, w_2] = [(+1,+1), (-1,-1), (+1,+1), (-1,-1), (+1,+1),(+1,+1)].\]
\end{remark}

In enumeration it is often advantageous and interesting to work with generating functions.
We will consider the following generating functions
\begin{align*}
    G(x; \bA^{(r)}) &= \sum_{n \geq 0} a^{(r)}_n x^{n} & G(x; \bD^{(r)}) &= \sum_{n \geq 0} d^{(r)}_n x^{n}\\ 
    G(x; \bB^{(r)}) &= \sum_{n \geq 0} b^{(r)}_n x^{n} & G(x; \bE^{(r)}) &= \sum_{n \geq 0} e^{(r)}_n x^{n}\\
    G(x; \bC^{(r)}) &= \sum_{n \geq 0} c^{(r)}_n x^{n} & G(x; \bF^{(r)}) &= \sum_{n \geq 0} f^{(r)}_n x^{n}
\end{align*}
which are the ordinary generating functions for our sequences.
Also for any $r \geq 0$, $v \in \calS_r$, and $\bL \in \{\bA^{(r)}, \bB^{(r)}, \bC^{(r)}, \bD^{(r)}, \bE^{(r)}, \bF^{(r)}\}$ we let  $\bL (v) \subseteq L$ be the sublanguage of walks which start with $v$.
Furthermore, we let
\[G(x; \bL (v)) = \sum_{n \geq 0} |\bL (v) \cap \calS_d^{2n}| x^n\]
be the generating function for these walks with a given first step.

We now review some facts on power series which can be found in~\cite{Stanley}.
A formal power series $G(x)$ is \emph{algebraic} if it satisfies a nontrivial polynomial equation.
A \emph{holonomic} sequence is a sequence that satisfies a recurrence relation with polynomial coefficients.
That is, a sequence $\{t_n\}_{n \geq 0}$ such that
\[p_j(n) t_{n+j} + p_{j-1}(n) t_{n+j-1} + \cdots + p_0(n) t_n = 0\]
for some $j$ and polynomials $p_i(n)$ not all of which are equal to zero.
The sequence of coefficients of an algebraic formal power series is holonomic.

From a generating function one can find the asymptotics of its sequence of coefficients.
Methods for obtaining these asymptotic expressions are well established and can be found in the text~\cite{analytic}.
We write $f(n) \sim g(n)$ to mean that
\[\lim_{n \to \infty} \frac{f(n)}{g(n)} = 1\]
which is the standard notation.
Consider a generating function $G(z) = \sum_{n \geq 0} g_n z^n$ viewed as a complex analytic function and let $\rho$ be its unique singularity closest to the origin.
Assume $G(z)$ approaches $C \cdot (1 - \frac{z}{\rho})^{-\alpha}$ as $z \to \rho$ for some constant $C$ where $\alpha$ is neither $0$ nor a negative integer.
This is the situation that will be relevant to use, and under these assumptions
\[g_n \sim \frac{C\rho^{-n} n^{\alpha-1}}{\Gamma(\alpha)}\]
where $\Gamma(\alpha)$ denotes the Gamma function.

Let us now show using formal language and generating function theory that each of our generating functions are algebraic from which it follows each of the sequences is holonomic.
Our goal will be to further understand these generating functions and find the recurrence relations that the sequences obey.
The Chomsky--Sch\"utzenberger enumeration theorem~\cite{CS} (see~\cite{KS,Pan} for proofs) says that the length generating function of an unambigous context free language is algebraic.
A deterministic context free language is a language recognized by a deterministic push down automata (DPDA).
Ginsburg and Greibach~\cite{GG} have shown that a deterministic context free language is an unambiguous context free language and that when such a language is intersected with a regular language it remains a deterministic context free language.

\begin{figure}
    \centering
    \begin{tikzpicture}[node distance=4cm,shorten >=2pt]
    \node[state, initial, accepting] (A) {};
    \node[state] (B)[right=of A] {};
    \path[->] (A) edge[bend left=45] node[above, text width=5cm,align=center]{$(v,1),Z_0 \to Z_0 U$\\$(v,-1),Z_0 \to Z_0 D$} (B);
    \path[->] (B) edge[bend left=45] node[below]{$\epsilon,Z_0 \to Z_0$} (A);
    \path[->] (B) edge[loop right] node[text width=3cm,align=left] {$(v,1), U \to UU$\\$(v,1), D \to \epsilon$\\$(v,-1), U \to \epsilon$\\$(v,-1), D \to DD$} ();
    \end{tikzpicture}
    \caption{A DPDA recognizing $\bA^{(r)}$ where $v$ represents any vector in $\{\pm 1\}^r$.}
    \label{fig:DPDA_A}
\end{figure}

\begin{figure}
    \centering
    \begin{tikzpicture}[node distance=4cm,shorten >=2pt]
    \node[state, initial, accepting] (A) {};
    \node[state] (B)[right=of A] {};
    \path[->] (A) edge[bend left=45] node[above, text width=5cm,align=center]{$(v,1),Z_0 \to Z_0 U$} (B);
    \path[->] (B) edge[bend left=45] node[below]{$\epsilon,Z_0 \to Z_0$} (A);
    \path[->] (B) edge[loop right] node[text width=3cm,align=left] {$(v,1), U \to UU$\\$(v,-1), U \to \epsilon$} ();
    \end{tikzpicture}
    \caption{A DPDA recognizing $\bD^{(r)}$ where $v$ represents any vector in $\{\pm 1\}^r$.}
    \label{fig:DPDA_D}
\end{figure}

\begin{theorem}
Let $\bL \in \{\bA^{(r)}, \bB^{(r)}, \bC^{(r)}, \bD^{(r)}, \bE^{(r)}, \bF^{(r)}\}$ be one of our languages,
then $L$ is a deterministic context free language.
Hence, $G(x;L)$ is algebraic and its sequence of coefficients is holomonic. 
\label{thm:CFL}
\end{theorem}

\begin{proof}
It suffices to prove that each of $\bA^{(r)}$ and $\bD^{(r)}$ is recognized by a DPDA.
Both of the languages $\bX^{(r)}$ and $\bY^{(r)}$ consisting of walks which respectively avoid the backtracking pattern $[v,-v]$ as well as which avoid the repeat pattern $[v,v]$ (but can end anywhere) are regular.
Since $\bB^{(r)} = \bA^{(r)} \cap \bX^{(r)}$, $\bC^{(r)} = \bA^{(r)} \cap \bY^{(r)}$, $\bE^{(r)} = \bD^{(r)} \cap \bX^{(r)}$, and $\bF^{(r)} = \bD^{(r)} \cap \bY^{(r)}$ it will follow that each of these will be deterministic context free languages.
By the Chomsky--Sch\"utzenberger enumeration theorem all the generating functions will be algebraic.
Therefore each of the sequences must be holonomic.
DPDAs recognizing the languages $\bA^{(r)}$ and $\bD^{(r)}$ are shown in Figure~\ref{fig:DPDA_A} and Figure~\ref{fig:DPDA_D} respectively.
\end{proof}

\begin{remark}
Theorem~\ref{thm:CFL} essentially follows from the definition of our languages.
Its purpose is to make rigorous the fact that generating functions are algebraic and that their sequences of coefficients are holonomic.
It is worth noting that there are other techniques that can also be used to obtain the algebraicity of the generating functions (e.g. the vectorial kernel method).
We also note the method above the flexible to other situations (e.g. intersecting with a different regular language coming from a pattern or patterns).
\end{remark}

Now that we know our generating functions are algebraic we turn our attention to describing these generating functions and their coefficient as explicitly as possible.
The values of $a_n^{(r)}$ and $F(x; \bA^{(r)})$ can be found without difficulty.
We have that
\begin{equation}
    a^{(r)}_n = 2^{2nr} \binom{2n}{n}
    \label{eq:a}
\end{equation}
since any walk in $\bA^{(r)}$ is
\[[v_1, v_2, \dots, v_{2n}]\]
where for coordinates $1 \leq i \leq r$ in $v_1, \dots, v_{2n}$ form any binary string and the last coordinate makes a balanced binary string.
Hence, $a^{(r)}_n$ satisfies the recurrence
\begin{equation}
    na^{(r)}_n = 2^{2r+1}(2n-1)a^{(r)}_{n-1}
\end{equation}
of the form guaranteed from being holonomic.
It follows that
\begin{equation}
    G(x;\bA^{(r)}) = \frac{1}{\sqrt{1 - 2^{2r+2}x}} = \frac{1}{\sqrt{1 - 4(4^rx)}}
    \label{eq:genA}
\end{equation}
is the expression for the generating function.
By similar arguments we find that
\begin{equation}
    d^{(r)}_n = \frac{2^{2nr}}{n+1}\binom{2n}{n} = 2^{2nr}C_n
    \label{eq:d}
\end{equation}
where $C_n = \frac{1}{n+1} \binom{2n}{n}$ is the $n$th Catalan number.
Also,
\begin{equation}
    (n+1)d^{(r)}_n = 2^{2r+1}(2n-1)d^{(r)}_{n-1}\end{equation}
and
\begin{equation}
    G(x; \bD^{(r)}) = \frac{1 - \sqrt{1 - 4(4^rx)}}{2}
    \label{eq:genD}
\end{equation}
give the recurrence and generating function for $d^{(r)}_n$

\section{Formulas and recurrences}
\label{sec:formula}

In this section we find formulas for each of the sequences $\{b^{(r)}_n\}_{n \geq 0 }$, $\{c^{(r)}_n\}_{n \geq 0 }$, $\{e^{(r)}_n\}_{n \geq 0 }$, and $\{f^{(r)}_n\}_{n \geq 0 }$ consisting of a sum of hypergeometric terms and as an evaluation of a hypergeometric function.
Zeilberger's algorithm~\cite{Z90,Z91} is then applied to obtain a recurrence.

\subsection{Ending on a hyperplane}

For a given integer $n$, an \emph{integer composition} of $n$ is an ordered sequence of positive integers which sum to $n$.
The notion of an integer composition will be essential to the proofs of the next two theorems.
We will write $\alpha \vDash n$ to denote that $\alpha$ is an integer composition of $n$.
The \emph{length} of an integer composition $\alpha$ is the number of elements in the sequence and is denoted by $\ell(\alpha)$.
For example, $\alpha = (1,1,3)$ and $\beta = (1,3,1)$ are two distinct integer compositions of $5$ both of length $\ell(\alpha) = \ell(\beta) = 3$.
A known fact about integer compositions which we will use is that there are $\binom{n-1}{k-1}$ integer compositions of $n$ with length equal to $k$.

\begin{theorem}
For any $r \geq 1$ and  $n \geq 1$ we have
\begin{align*}
   b^{(r)}_n &= 2\sum_{k=1}^n (2^r-1)^{2k-2}2^{r(2n-2k+1)} \binom{n-1}{k-1} \left((2^r-1) \binom{n-1}{k-1} + 2^r \binom{n-1}{k-2} \right)\\
   &= 2\left(2^{2rn} - 2^{r(2n-1)}\right) \pFq{3}{2}{-n, -n+1, 2^r n -n +1}{1, 2^r n - n}{(2^r-1)^2 \left(\frac{1}{2}\right)^{2r}}
\end{align*}
which satisfies the recurrence
\[nb^{(r)}_n = 2\left((2^{2r+1} - 2^{r+1} +1)(n-1) + 2^{2r} - 2^r\right)b^{(r)}_{n-1} - (2^{r+1}-1)^2(n-2)b^{(r)}_{n-2}\]
for $n \geq 3$ with $b^{(r)}_1 = 2^{r+1}(2^r - 1)$ and
\[b^{(r)}_2 = 2^{3r+1}(2^r-1) + 2^{2r+1}(2^r-1)^2 + 2^{r+1}(2^r-1)^3.\]


\label{thm:B}
\end{theorem}
\begin{proof}
Recall that $b^{(r)}_n$ counts the number of walks of length $2n$ which avoid backtracking and end with $x_{r+1} = 0$.
Consider the projection of such a walk onto the $x_{r+1}$-axis.
This projection can be encoded by a word with up steps $U$ and down steps $D$ corresponding to if step in the walk had the last coordinate $+1$ or $-1$ respectively.
Let us assume the word starts with a $U$. 
The case that starts with a $D$ is completely analogous.
The lengths of the runs of $U$'s and $D$'s will give two integer compositions $\alpha \vDash n$ and $\beta \vDash n$ respectively where $\ell = \ell(\beta) \in \{k-1, k\}$ when $\ell(\alpha) = k$.
The vector corresponding to each $U$ and $D$ has the last coordinate determined and then has either $2^r$ or $2^r-1$ choices for which $v \in \{\pm 1\}^r$ it could have came from making the prefix of the first $r$ coordinates of the vector.
Looking at only the $U$'s, we see that the $U's$ with only $2^r - 1$ choices are those in positions $\alpha_1+1, \alpha_1 + \alpha_2+1, \dots, \alpha_1 + \cdots + \alpha_{k-1}+1$.
Looking at only the $D$'s, we see that the $D$'s with only $2^r - 1$ choices are those in positions $1, \beta_1+1, \beta_1 + \beta_2+1, \dots, \beta_1 + \cdots + \beta_{\ell-1}+1$.
Thus we find that
\begin{align*}
     b^{(r)}_n &= 2\sum_{k = 1}^n \sum_{\substack{\alpha \vDash n \\ \ell(\alpha) = k}} (2^r-1)^{k-1}2^{r(n-k+1)}\left( \sum_{\substack{\beta \vDash n \\ \ell(\beta) = k}}(2^r-1)^k 2^{r(n-k)} + \sum_{\substack{\beta \vDash n \\ \ell(\beta) = k-1}}(2^r-1)^{k-1} 2^{r(n-k+1)}\right) \\
     &= 2\sum_{k = 1}^n (2^r-1)^{k-1}2^{r(n-k+1)} \binom{n-1}{k-1} \left( (2^r-1)^k 2^{r(n-k)} \binom{n-1}{k-1} + (2^r-1)^{k-1} 2^{r(n-k+1)} \binom{n-1}{k-2} \right)\\
     &= 2\sum_{k=1}^n (2^r-1)^{2k-2}2^{r(2n-2k+1)} \binom{n-1}{k-1} \left((2^r-1) \binom{n-1}{k-1} + 2^r \binom{n-1}{k-2} \right)
\end{align*}
Once we have $b^{(r)}_n$ expressed in this way as a sum over hypergeometric terms we may find the hypergeometric evaluation and recurrence using automated methods.
In the appendix we show how this computation can be performed in Maple.

\end{proof}

\begin{example}
For $r=1$ and $n = 4$ along with the two integer compositions $(1,2,1)$ and $(2,2)$ we have $2^5 = 32$ lattice walks in the plane contributing to the count of $a^{(1)}_4$.
Projecting onto the $x_2$-axis all such walks will be either 
\[[{}^2U,{}^1D, {}^2D,{}^1U,{}^2U,{}^1D,{}^2D,{}^1U]\] 
or
\[[{}^2D,{}^1U,{}^2U,{}^1D,{}^2D,{}^1U,{}^2U,{}^1D].\]
Here the superscripts indicate how many choices we have for each step as a $2$-dimensional walk.
We let
\[ \{\pm 1\}^2 = \{NE := (1,1), NW := (-1,1), SE := (1,-1), SW := (-1,-1)\}.\]
Accounting for the symmetry of reflection over the $x_1$-axis and the $x_2$-axis (i.e. exchanging $E$ and $W$ or exchanging $N$ and $S$) there are $8$ walks each starting with $NE$ which are shown both as words and graphed in the plane in Figure~\ref{fig:PlaneWalks}.
\end{example}

\begin{figure}
    \centering
\begin{tikzpicture}[scale=0.35]

\begin{scope}[shift={(-17.5,0)}]
\draw[step=1] (-5,-2) grid (9,2);
\draw[thick] (-1,0)--(9,0);
\draw[thick] (0,-2)--(0,2);
\draw[thick, blue, -{latex}] (0,0)--(1,1)--(2,0)--(3,-1)--(4,0)--(5,1)--(6,0)--(7,-1)--(8,0);
\node (x) at (2,-3) {$[NE,SE,SE,NE,NE,SE,SE,NE]$};
\end{scope}

\begin{scope}[shift={(2.5,0)}]
\draw[step=1] (-5,-2) grid (9,2);
\draw[thick] (-1,0)--(9,0);
\draw[thick] (0,-2)--(0,2);
\draw[thick, blue, -{latex}] (0,0)--(1,1)--(2,0)--(3,-1)--(4,0)--(5,1)--(6,0)--(5,-1)--(4,0);
\node (x) at (2,-3) {$[NE,SE,SE,NE,NE,SE,SW,NW]$};
\end{scope}

\begin{scope}[shift={(-17.5,-7)}]
\draw[step=1] (-5,-2) grid (9,2);
\draw[thick] (-1,0)--(9,0);
\draw[thick] (0,-2)--(0,2);
\draw[thick, blue, -{latex}] (0,0)--(1,1)--(2,0)--(3,-1)--(4,0)--(3,1)--(2,0)--(3,-1)--(4,0);
\node (x) at (2,-3) {$[NE,SE,SE,NE,NW,SW,SE,NE]$};
\end{scope}

\begin{scope}[shift={(2.5,-7)}]
\draw[step=1] (-5,-2) grid (9,2);
\draw[thick] (-1,0)--(9,0);
\draw[thick] (0,-2)--(0,2);
\draw[thick, blue, -{latex}] (0,0)--(1,1)--(2,0)--(1,-1)--(0,0)--(1,1)--(2,0)--(3,-1)--(4,0);
\node (x) at (2,-3) {$[NE,SE,SW,NW,NE,SE,SE,NE]$};
\end{scope}

\begin{scope}[shift={(-17.5,-14)}]
\draw[step=1] (-5,-2) grid (9,2);
\draw[thick] (-1,0)--(9,0);
\draw[thick] (0,-2)--(0,2);
\draw[thick, blue, -{latex}] (0,0)--(1,1)--(2,0)--(3,-1)--(4,0)--(3,1)--(2,0)--(1,-1)--(0,0);
\node (x) at (2,-3) {$[NE,SE,SE,NE,NW,SW,SW,NW]$};
\end{scope}

\begin{scope}[shift={(2.5,-14)}]
\draw[step=1] (-5,-2) grid (9,2);
\draw[thick] (-1,0)--(9,0);
\draw[thick] (0,-2)--(0,2);
\draw[thick, blue, -{latex}] (0,0)--(1,1)--(2,0)--(1,-1)--(0,0)--(1,1)--(2,0)--(1,-1)--(0,0);
\node (x) at (2,-3) {$[NE,SE,SW,NW,NE,SE,SW,NW]$};
\end{scope}

\begin{scope}[shift={(-17.5,-21)}]
\draw[step=1] (-5,-2) grid (9,2);
\draw[thick] (-1,0)--(9,0);
\draw[thick] (0,-2)--(0,2);
\draw[thick, blue, -{latex}] (0,0)--(1,1)--(2,0)--(1,-1)--(0,0)--(-1,1)--(-2,0)--(-1,-1)--(0,0);
\node (x) at (2,-3) {$[NE,SE,SW,NW,NW,SW,SE,NE]$};
\end{scope}

\begin{scope}[shift={(2.5,-21)}]
\draw[step=1] (-5,-2) grid (9,2);
\draw[thick] (-1,0)--(9,0);
\draw[thick] (0,-2)--(0,2);
\draw[thick, blue, -{latex}] (0,0)--(1,1)--(2,0)--(1,-1)--(0,0)--(-1,1)--(-2,0)--(-3,-1)--(-4,0);
\node (x) at (2,-3) {$[NE,SE,SW,NW,NW,SW,SW,NE]$};
\end{scope}
\end{tikzpicture}
\caption{Here are $8$ walks from which all $32$ walks in the plane avoiding backtracking corresponding to the integer compositions $(1,2,1)$ and $(2,2)$ can be obtained.}
\label{fig:PlaneWalks}
\end{figure}
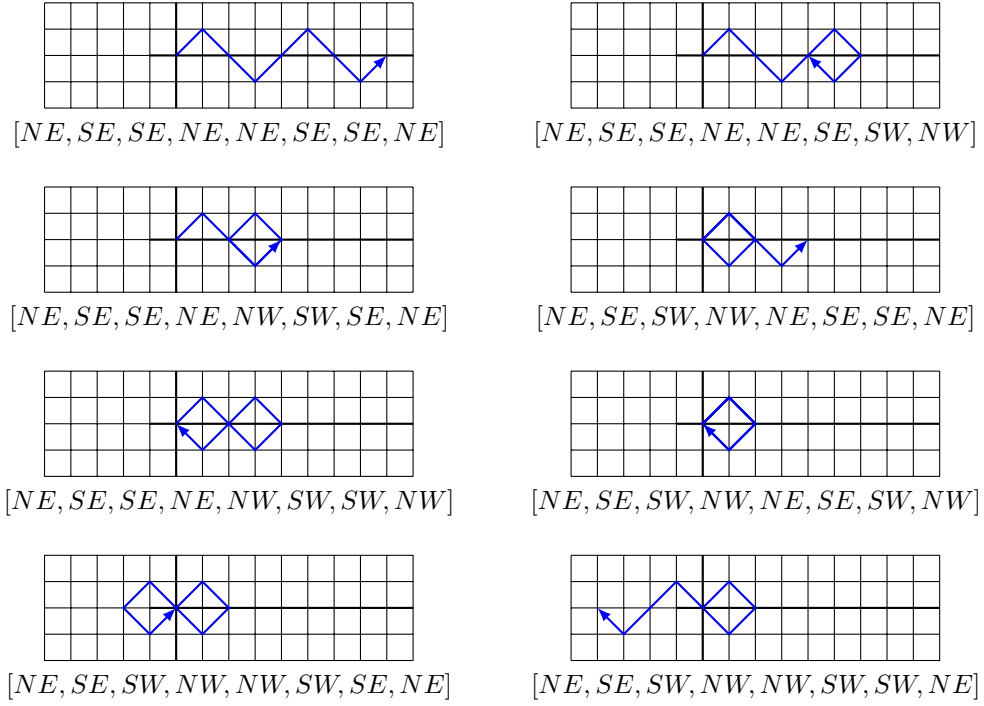

\begin{theorem}
For any $r \geq 1$ and $n \geq 1$ we have
\begin{align*}
c^{(r)}_n &= 2\sum_{k=1}^n (2^r-1)^{2n-2k}2^{r(2k-1)} \binom{n-1}{k-1} \left(2^r\binom{n-1}{k-1} + (2^r-1) \binom{n-1}{k-2} \right)\\
&= 2^{2r+1} (2^r-1)^{2n-2} \pFq{3}{2}{-n, -n+1, -2^r n + 1}{1, -2^r n}{\frac{2^{2r}}{(2^r-1)^2}}
\end{align*}
which satisfies the recurrence
\[nc^{(r)}_n = 2\left((2^{2r+1} - 2^{r+1} + 1)(n-1) + 2^{2r} - 2^r\right)c^{(r)}_{n-1} - (2^{r+1} - 1)^2 (n-2)c^{(r)}_{n-2}\]
for $n \geq 3$ with $c^{(r)}_1 = 2^{2r+1}$ and
\[c^{(r)}_2 = 2^{4r+1} + 2^{3r+1}(2^r-1) + 2^{2r+1}(2^r-1)^2.\]
\label{thm:C}
\end{theorem}
\begin{proof}
Recall that $c^{(r)}_n$ counts the number of walks of length $2n$ which avoid consecutive repeated steps and end with $x_{r+1} = 0$.
Proceeding similarly to the proof of Theorem~\ref{thm:B}, we consider the projection of such a walk onto the $x_{r+1}$-axis.
Again this projection can be encoded by a word with up steps $U$ and down steps $D$ meaning the corresponding step in the walk had the last coordinate $+1$ or $-1$ respectively.
Let us assume the word starts with a $U$ since the case that starts with a $D$ is completely analogous.
The lengths of the runs of $U$'s and $D$'s will give two integer compositions $\alpha \vDash n$ and $\beta \vDash n$ respectively where $\ell = \ell(\beta) \in \{k-1, k\}$ when $\ell(\alpha) = k$.
The vector corresponding to each $U$ and $D$ has the last coordinate determined and then has either $2^r$ or $2^r-1$ choices for which $v \in \{\pm 1\}^r$ it could have came from making the prefix of the first $r$ coordinates.
Looking at only the $U$'s, we see that the $U's$ with $2^r$ choices are those in positions $1, \alpha_1+1, \alpha_1 + \alpha_2+1, \dots, \alpha_1 + \cdots + \alpha_{k-1}+1$.
Looking at only the $D$'s, we see that the $D$'s with $2^r$ choices are those in positions $1, \beta_1+1, \beta_1 + \beta_2+1, \dots, \beta_1 + \cdots + \beta_{\ell-1}+1$.
So, it follows that
\begin{align*}
     c^{(r)}_n &= 2\sum_{k = 1}^n \sum_{\substack{\alpha \vDash n \\ \ell(\alpha) = k}} (2^r-1)^{n-k}2^{rk}\left( \sum_{\substack{\beta \vDash n \\ \ell(\beta) = k}}(2^r-1)^{n-k} 2^{rk} + \sum_{\substack{\beta \vDash n \\ \ell(\beta) = k-1}}(2^r-1)^{n-k+1} 2^{r(k-1)}\right) \\
     &= 2\sum_{k = 1}^n (2^r-1)^{n-k}2^{rk} \binom{n-1}{k-1} \left( (2^r-1)^{n-k} 2^{rk} \binom{n-1}{k-1} + (2^r-1)^{n-k+1} 2^{r(k-1)} \binom{n-1}{k-2} \right)\\
     &= 2\sum_{k=1}^n (2^r-1)^{2n-2k}2^{r(2k-1)} \binom{n-1}{k-1} \left(2^r\binom{n-1}{k-1} + (2^r-1) \binom{n-1}{k-2} \right)
\end{align*}
Notice the only difference with Theorem~\ref{thm:B} is that exponents of $(2^r-1)$ and $2^r$ are changed.
Once we have $c^{(r)}_n$ expressed in this way as a sum over hypergeometric terms we may find the hypergeometric evaluation and recurrence using automated methods.
In the appendix we show how this computation can be performed in Maple.
\end{proof}

\subsection{Ending on a hyperplane and restricted to a halfspace}
\begin{theorem}
For any $r \geq 1$ and $n \geq 1$ we have
\begin{align*}
e^{(r)}_n &= \frac{1}{n}\sum_{k=1}^n (2^r -1)^{2k-1}2^{r(2n-2k+1)}  \binom{n}{k}\binom{n}{k-1}\\
&= \left(2^{2rn} - 2^{r(2n-1)}\right) \pFq{2}{1}{-n, -n+1}{2}{(2^r-1)^2 \left(\frac{1}{2}\right)^{2r}}
\end{align*}
which satisfies the recurrence
\[(n+1)e^{(r)}_n = (2^{2r+1} - 2^{r+1} +1)(2n-1)e^{(r)}_{n-1} - (2^{r+1}-1)^2 (n-2)e^{(r)}_{n-2}\]
for $n \geq 3$ with $e^{(r)}_1 = 2^r(2^r - 1)$ and $e^{(r)}_2 = 2^{3r}(2^r-1) + 2^r(2^r-1)^3$.
\label{thm:E}
\end{theorem}

\begin{proof}
Recall that $e^{(r)}_n$ counts the number of walks of length $2n$ which avoid backtracking and end with $x_{r+1} = 0$ while staying the half space defined by $x_{r+1} \geq 0$.
In the last coordinate we must have a Dyck path.
We partition Dyck paths of semilength $n$ by number of peaks.
The Narayana number
\[N(n,k) = \frac{1}{n} \binom{n}{k}\binom{n}{k-1}\]
gives the number of Dyck paths of semilength $n$ with $k$ peaks (and hence $k-1$ valleys).
For the first $r$ coordinates with have $2^r$ choices for each step except for the steps directly after a peak or valley for which we only have $(2^r-1)$ choices.
Once we have $e^{(r)}_n$ expressed in this way as a sum over hypergeometric terms we may find the hypergeometric evaluation and recurrence using automated methods.
In the appendix we show how this computation can be performed in Maple.
\end{proof}

\begin{theorem}
For any $r \geq 1$ and $n \geq 1$ we have
\begin{align*}
f^{(r)}_n &=  \frac{1}{n}\sum_{k=1}^n (2^r -1)^{2n - 2k}2^{2rk} \binom{n}{k}\binom{n}{k-1}\\
&= 2^{2r} (2^r-1)^{2n-2} \pFq{2}{1}{-n, -n+1}{2}{\frac{2^{2r}}{(2^r-1)^2}}
\end{align*}
which satisfies the recurrence
\[(n+1)f^{(r)}_n = (2^{2r+1} - 2^{r+1} +1)(2n-1)f^{(r)}_{n-1} - (2^{r+1}-1)^2 (n-2)f^{(r)}_{n-2}\]
for $n \geq 3$ with $f^{(r)}_1 = 2^{2r}$ and $f^{(r)}_2 =2^{4r} + 2^{2r}(2^r-1)^2$.
\label{thm:F}
\end{theorem}

\begin{proof}
Recall that $f^{(r)}_n$ counts the number of walks of length $2n$ which avoid the repeat pattern and end with $x_{r+1} = 0$ while staying the half space defined by $x_{r+1} \geq 0$.
In the last coordinate we must have a Dyck path.
We again partition Dyck paths of semilength $n$ by number of peaks given by Narayana numbers
\[N(n,k) = \frac{1}{n} \binom{n}{k}\binom{n}{k-1}\]
for each $1 \leq k \leq n$.
For the first $r$ coordinates with have $2^r$ for only the steps directly after a peak or valley as well as for the first step.
At all other steps we have $(2^r-1)$ choices for the first $r$ coordinates.
Once we have $f^{(r)}_n$ expressed in this way as a sum over hypergeometric terms we may find the hypergeometric evaluation and recurrence using automated methods.
In the appendix we show how this computation can be performed in Maple.
\end{proof}

\begin{example}
For $r=1$ and $n = 4$ consider walks contributing to $f^{(1)}_4$ projecting onto the Dyck path $[{}^2U,{}^1U,{}^1U,{}^2D,{}^2U,{}^2D,{}^1D,{}^1D]$ where $U$ and $D$ are up and down steps respectively and superscripts indicated the number of choices for each step.
There are $2^4 = 16$ such paths because this Dyck path has $2$ peaks and $1$ valley.
We let
\[ \{\pm 1\}^2 = \{NE := (1,1), NW := (-1,1), SE := (1,-1), SW := (-1,-1)\}\]
and accounting for symmetry we may assume $NE$ is the first step.
This results in $8$ walks shown in Figure~\ref{fig:F}.
\end{example}

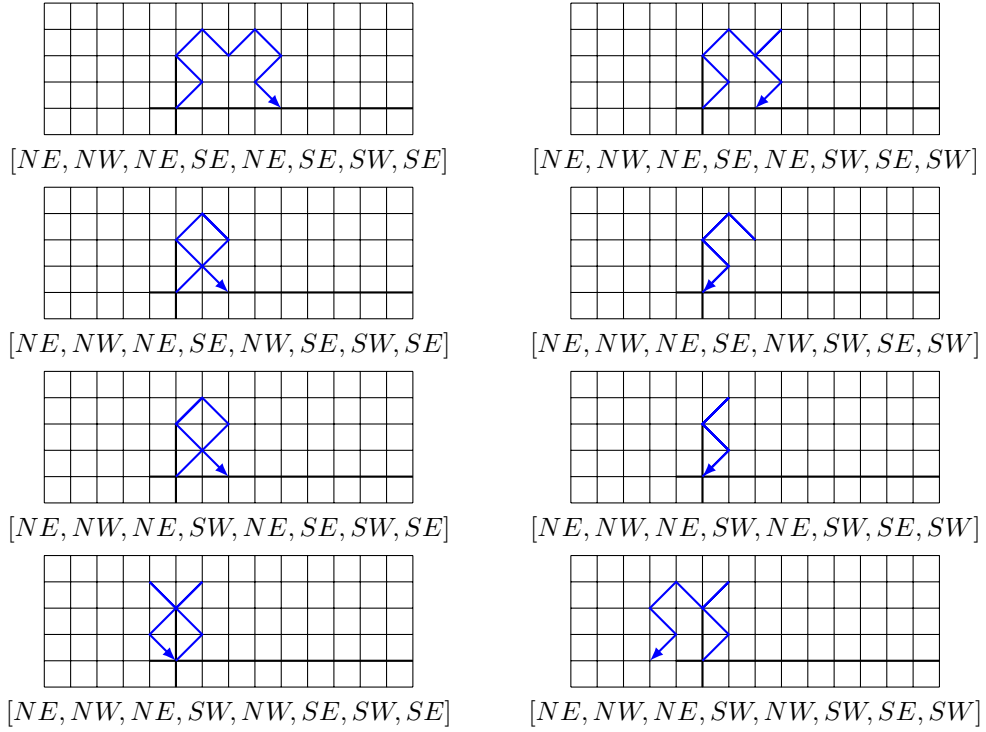
\begin{figure}
    \centering
\begin{tikzpicture}[scale=0.35]

\begin{scope}[shift={(-17.5,0)}]
\draw[step=1] (-5,-1) grid (9,4);
\draw[thick] (-1,0)--(9,0);
\draw[thick] (0,-1)--(0,2);
\draw[thick, blue, -{latex}] (0,0)--(1,1)--(0,2)--(1,3)--(2,2)--(3,3)--(4,2)--(3,1)--(4,0);
\node (x) at (2,-2) {$[NE,NW,NE,SE,NE,SE,SW,SE]$};
\end{scope}

\begin{scope}[shift={(2.5,0)}]
\draw[step=1] (-5,-1) grid (9,4);
\draw[thick] (-1,0)--(9,0);
\draw[thick] (0,-1)--(0,2);
\draw[thick, blue, -{latex}] (0,0)--(1,1)--(0,2)--(1,3)--(2,2)--(3,3)--(2,2)--(3,1)--(2,0);
\node (x) at (2,-2) {$[NE,NW,NE,SE,NE,SW,SE,SW]$};
\end{scope}

\begin{scope}[shift={(-17.5,-7)}]
\draw[step=1] (-5,-1) grid (9,4);
\draw[thick] (-1,0)--(9,0);
\draw[thick] (0,-1)--(0,2);
\draw[thick, blue, -{latex}] (0,0)--(1,1)--(0,2)--(1,3)--(2,2)--(1,3)--(2,2)--(1,1)--(2,0);
\node (x) at (2,-2) {$[NE,NW,NE,SE,NW,SE,SW,SE]$};
\end{scope}

\begin{scope}[shift={(2.5,-7)}]
\draw[step=1] (-5,-1) grid (9,4);
\draw[thick] (-1,0)--(9,0);
\draw[thick] (0,-1)--(0,2);
\draw[thick, blue, -{latex}] (0,0)--(1,1)--(0,2)--(1,3)--(2,2)--(1,3)--(0,2)--(1,1)--(0,0);
\node (x) at (2,-2) {$[NE,NW,NE,SE,NW,SW,SE,SW]$};
\end{scope}

\begin{scope}[shift={(-17.5,-14)}]
\draw[step=1] (-5,-1) grid (9,4);
\draw[thick] (-1,0)--(9,0);
\draw[thick] (0,-1)--(0,2);
\draw[thick, blue, -{latex}] (0,0)--(1,1)--(0,2)--(1,3)--(0,2)--(1,3)--(2,2)--(1,1)--(2,0);
\node (x) at (2,-2) {$[NE,NW,NE,SW,NE,SE,SW,SE]$};
\end{scope}

\begin{scope}[shift={(2.5,-14)}]
\draw[step=1] (-5,-1) grid (9,4);
\draw[thick] (-1,0)--(9,0);
\draw[thick] (0,-1)--(0,2);
\draw[thick, blue, -{latex}] (0,0)--(1,1)--(0,2)--(1,3)--(0,2)--(1,3)--(0,2)--(1,1)--(0,0);
\node (x) at (2,-2) {$[NE,NW,NE,SW,NE,SW,SE,SW]$};
\end{scope}

\begin{scope}[shift={(-17.5,-21)}]
\draw[step=1] (-5,-1) grid (9,4);
\draw[thick] (-1,0)--(9,0);
\draw[thick] (0,-1)--(0,2);
\draw[thick, blue, -{latex}] (0,0)--(1,1)--(0,2)--(1,3)--(0,2)--(-1,3)--(0,2)--(-1,1)--(0,0);
\node (x) at (2,-2) {$[NE,NW,NE,SW,NW,SE,SW,SE]$};
\end{scope}

\begin{scope}[shift={(2.5,-21)}]
\draw[step=1] (-5,-1) grid (9,4);
\draw[thick] (-1,0)--(9,0);
\draw[thick] (0,-1)--(0,2);
\draw[thick, blue, -{latex}] (0,0)--(1,1)--(0,2)--(1,3)--(0,2)--(-1,3)--(-2,2)--(-1,1)--(-2,0);
\node (x) at (2,-2) {$[NE,NW,NE,SW,NW,SW,SE,SW]$};
\end{scope}
\end{tikzpicture}
\caption{Here are $8$ walks from which all $16$ walks in the plane avoiding consecutive steps projecting up the Dyck path $[U,U,U,D,U,D,D,D]$ can be obtained.}
\label{fig:F}
\end{figure}

\section{Generating functions and asymptotics}
\label{sec:gen}

In this section we give a formula for each of the generating functions $G(x; \bB^{(r)})$, $G(x; \bC^{(r)})$, $G(x; \bE^{(r)})$, and $G(x; \bF^{(r)})$.

\begin{lemma}
If $r \geq 0$ and $v = (v',1) \in \calS_r$, then $G(x; \bE^{(r)}(v)) = \frac{1}{2^{r}}\left(G(x; \bE^{(r)}) - 1\right)$.
\label{lem:startE}
\end{lemma}
\begin{proof}
Let $\epsilon_i: \calS_r \to \calS_r$ be defined by 
\[(y_1, \dots, y_{i-1}, y_i, y_{i+1} \dots, y_{r+1}) \mapsto (y_1, \dots, y_{i-1}, -y_i, y_{i+1} \dots, y_{r+1}) \]
for each $1 \leq i \leq r$.
It is clear $\epsilon_i$ is an involution, and hence a bijection.
We have that $[v_1, v_2, \dots, v_{2n}] \in \bE^{(r)}$ if and only if $[\epsilon_i(v_1), \epsilon_i(v_2), \dots, \epsilon_i(v_{2n})] \in \bE^{(r)}$.
This is because $v = -u$ if and only if $\epsilon_i(v) = - \epsilon_i(u)$ and because the sum of the last coordinate is preserved.
Now given any $v = (v', 1) \in \calS_r$ we can obtain any other $u = (u',1) \in \calS_r$ by $\epsilon_{i_\ell} \circ \epsilon_{i_{\ell-1}} \circ \cdots \circ \epsilon_{i_1}$ for any some sequence $i_1, i_2, \dots, i_{\ell}$ (i.e. the coordinates where $v$ and $u$ differ).
It follows that $G(x; \bE^{(r)}(v)) = G(x; \bE^{(r)}(u))$.
The lemma follows since
\[G(x; \bE^{(r)})-1 = \sum_{v = (v',1) \in \calS_r} G(x; \bE^{(r)}(v))\]
as all nonempty walks must start with some $v = (v',1) \in \calS_r$ while $\bE^{(r)}(v) \cap \bE^{(r)}(u) = \varnothing$ for $v \neq u$.
\end{proof}

\begin{lemma}
If $r \geq 0$ and $v \in \calS_r$, then $G(x; \bB^{(r)}(v)) = \frac{1}{2^{r+1}}\left(G(x; \bB^{(r)}) - 1\right)$.\label{lem:startB}
\end{lemma}
\begin{proof}
The proof is very similar to the proof of Lemma~\ref{lem:startE}.
The only essential difference is that a walk in $\bB^{(r)}$ can begin with any $v \in \calS_r$ and hence we are also allowed the use the similarly defined $\epsilon_{r+1}: \calS_r \to \calS_r$.
\end{proof}

\begin{lemma}
For any $r, n \geq 1$ we have that $2^r b^{(r)}_n = (2^r-1)c^{(r)}_n$ and $2^re^{(r)}_n = (2^r-1)f^{(r)}_n$.
\label{lem:relseq}
\end{lemma}
\begin{proof}
Using Theorem~\ref{thm:B}, Theorem~\ref{thm:C}, Theorem~\ref{thm:E}, and Theorem~\ref{thm:F} we see that $\{b^{(r)}_n\}$ satisfies the same recurrence as $\{c^{(r)}_n\}$ while $\{e^{(r)}_n\}$ satisfies the same recurrence as $\{f^{(r)}_n\}$.
It remains only to check initial conditions.
We find that
\[2^r b^{(r)}_1 = 2^r(2^{r+1}(2^r - 1)) = (2^r-1)2^{2r+1} = (2^r-1)c^{(r)}_1\]
and also
\begin{align*}
2^r b^{(r)}_2 &= 2^r(2^{3r+1}(2^r-1) + 2^{2r+1}(2^r-1)^2 + 2^{r+1}(2^r-1)^3)\\
&= (2^r-1)(2^{4r+1} + 2^{3r+1}(2^r-1) + 2^{2r+1}(2^r-1)^2)\\
&= (2^r-1)c^{(r)}_2
\end{align*}
so the lemma is proven for $b^{(r)}_n$ and $c^{(r)}_n$.
In a similar way
\[2^re^{(r)}_1 = 2^r(2^r(2^r - 1)) = (2^r-1)(2^{2r}) = (2^r-1) f^{(r)}_1\]
and
\[2^r e^{(r)}_2 = 2^r(2^{3r}(2^r-1) + 2^r(2^r-1)^3) = (2^r-1)(2^{4r} + 2^{2r}(2^r-1)^2) = (2^r-1) f^{(r)}_2\]
which completes the proof of the lemma.
\end{proof}

\begin{theorem}
For any $r > 0$ we have
\begin{align*}
    G(x; \bB^{(r)}) &= \sqrt{\frac{1-x}{1-(2^{r+1}-1)^2x}}\\
    G(x; \bC^{(r)}) &=  \frac{2^r\sqrt{1-x} - \sqrt{1- (2^{r+1}-1)^2 x}}{ (2^r-1)\sqrt{1-(2^{r+1}-1)^2 x}}\\
    G(x; \bE^{(r)}) &= \frac{1 - x - \sqrt{\big(1-(2^{r+1}-1)^2x\big)\left(1-x\right)}}{2^{r+1}(2^r-1)x}\\
    G(x; \bF^{(r)}) &= \frac{1 - (2^{r+1}-1)x - \sqrt{\big(1-(2^{r+1}-1)^2x\big)(1-x)}}{2(2^r-1)^2x}
\end{align*}
while for $r = 0$ we have
\begin{align*}
    G(x; \bC^{(0)}) &= \frac{1+x}{1-x} &G(x; \bF^{(0)}) &= \frac{1}{1-x}
\end{align*}
and $G(x; \bB^{(0)}) = G(x; \bE^{(0)}) = 1$.
\label{thm:GF}
\end{theorem}

\begin{proof}
For $r=0$ the result is easy as there are no nonempty walks avoiding backtracking ending at the origin while $[+1,-1]^n$ and $[-1,+1]^n$ are the only walks avoiding consecutive steps ending at the origin with only the former confined to the nonnegative integers.
To establish the generating function identity for $\bE^{(r)}$ with $r > 1$ we split the walk where it first returns to $x_{r+1} = 0$. We note any walk in $\bE^{(r)}$ is:
\begin{enumerate}
    \item[(i)] empty,
    \item[(ii)] or of the form $[(u,1), w_1, (v,-1), w_2]$ such that $u,v \in \calS_{r-1}$ and $w_1, w_2 \in \bE^{(r)}$ where $w_1$ is nonempty and $w_2$ does not begin with $(-v,1)$,
    \item[(iii)] or of the form $[(u,1), (v,-1), w]$ such that $u, v \in \calS_{r-1}$ with $v \neq -u$ and $w \in \bE^{(r)}$ where $w$ does not begin with $(-v,1)$.
\end{enumerate}
Making use of Lemma~\ref{lem:startE} it follows that $G = G(x; \bE^{(r)})$ satisfies
\[G = 1 + 2^{2r}x(G-1)\left(1 + \frac{2^r-1}{2^r}\left(G-1\right)\right) + 2^r(2^r-1)x\left(1 + \frac{2^r-1}{2^r}\left(G-1\right)\right)\]
which is equivalent to quadratic equation
\[0 = 2^r(2^r-1)x(G-1)^2 + ((2^{2r} + (2^r-1)^2)x -1 )(G-1) + 2^r(2^r-1)x\]
that we can solve.
For $r > 0$ using the quadratic formula and making the correct choice of sign we have
\begin{align*}
    G - 1 &= \frac{1 - (2^{2r} + (2^r-1)^2)x - \sqrt{\left(\left(2^{2r}-(2^r-1)^2\right)^2x - 1\right)(x-1)}}{2^{r+1}(2^r-1)x}\\
        &= \frac{1 - (2^{2r} + (2^r-1)^2)x - \sqrt{\big(1-(2^{r+1}-1)^2x\big)(1-x)}}{2^{r+1}(2^r-1)x}
\end{align*}
which completes the proof for $G = G(x; \bE^{(r)})$ after adding $1$ to each side.
By Lemma~\ref{lem:relseq} it follows that
\[G(x;\bF^{(r)}) = 1 + \frac{2^r}{2^r-1}\left(G(x; \bE^{(r)})-1\right)\]
which can be used to conclude the theorem for $G(x; \bF^{(r)})$.

Now consider let's consider walks in $\bB^{(r)}$ which we again split according to where they first return to $x_{r+1} = 0$.
Such a walk must be:
\begin{enumerate}
    \item[(i)] empty,
    \item[(ii)] of the form $[(u,\pm 1), w_1, (v,\mp 1), w_2]$ such that $u, v \in \calS_{r-1}$, $\pm w_1 \in \bE^{(r)}$ is nonempty where $\pm 1$ is the last coordinate of the first step of the walk, and $w_2 \in \bB^{(r)}$ which does not start with $(-v, \pm 1)$
    \item[(iii)] of the form $[(u,\pm 1), (v,\mp 1), w]$ such that $u, v \in \calS_{r-1}$ with $v \neq -u$ and $w \in \bB^{(r)}$ which does not start with $(-v,\pm 1)$.
\end{enumerate}
Letting $H = G(x; \bB^{(r)})$ while still letting $G = G(x; \bE^{(r)})$ we have that
\[H = 1 + 2^{2r+1} x (G-1) \left(1 + \frac{2^{r+1}-1}{2^{r+1}} \left(H-1\right)\right) + 2^{r+1}(2^r-1)x\left(1 + \frac{2^{r+1}-1}{2^{r+1}} \left(H-1\right)\right)\]
where we have made use of Lemma~\ref{lem:startB}.
After rearranging we find
\[(1 + (2^{r+1} - 1)x - 2^r(2^{r+1}-1)G) \cdot H = 1 - x - 2^rxG\]
then solving for $H$ and substituting the formula we have previously found for $G$ we have
\begin{align*}
    H &= \frac{(2^{r+1}-1) - (2^{r+1}-1)x - \sqrt{(1 - (2^{r+1}-1)^2x)(1-x)}}{-1 + (2^{r+1}-1)^2x + (2^{r+1}-1)  \sqrt{(1 - (2^{r+1}-1)^2x)(1-x)}}\\
    &= \frac{\sqrt{1-x} \left((2^{r+1}-1)\sqrt{1-x} - \sqrt{1 - (2^{r+1}-1)^2x}\right)}{\sqrt{1 - (2^{r+1}-1)^2x} \left( - \sqrt{1 - (2^{r+1}-1)^2x} + (2^{r+1}-1)\sqrt{1-x}\right)}
\end{align*}
which derives the formula for $H = G(x; \bB^{(r)})$.
Last it remains to demonstrate the formula for $H = G(x; \bC^{(r)})$ and this can be readily done using Lemma~\ref{lem:relseq}.
\end{proof}

\begin{corollary}
For any $r \geq 1$, we have
\begin{align*}
    b^{(r)}_n &\sim (2^{r+1}-1)^{2n-1} \cdot \sqrt{\frac{(2^{r+1}-1)^2-1}{\pi n}} \\
    c^{(r)}_n &\sim \frac{2^r (2^{r+1}-1)^{2n-1}}{2^r-1} \cdot \sqrt{\frac{(2^{r+1}-1)^2-1}{\pi n}} \\
    e^{(r)}_n &\sim \frac{(2^{r+1}-1)^{2n+1}}{2^{r+2}(2^r-1)} \cdot \sqrt{\frac{(2^{r+1}-1)^2 - 1}{\pi n^3}} \\
    f^{(r)}_n &\sim \frac{(2^{r+1}-1)^{2n+1}}{2^{2}(2^r-1)^2} \cdot \sqrt{\frac{(2^{r+1}-1)^2 - 1}{\pi n^3}}
\end{align*}
as asymptotics for our sequences.
\label{cor:asym}
\end{corollary}

\begin{proof}
By Lemma~\ref{lem:relseq} we need only compute the asymptotics for $b^{(r)}_n$ and $e^{(r)}_n$.
For $b^{(r)}_n$ we have the generating function
\[G(z; \bB^{(r)}) = \sqrt{\frac{1-z}{1-(2^{r+1}-1)^2z}}\]
by Theorem~\ref{thm:GF} which has $\rho = \frac{1}{(2^{r+1}-1)^2}$ as its singularity closest to the origin.
As $z \to \rho$ we have that the generating function approaches
\[\sqrt{\frac{(2^{r+1}-1)^2 - 1}{(2^{r+1}-1)^2}} \cdot (1 - (2^{r+1}-1)^2x)^{-\alpha}\]
as a complex analytic function with $\alpha = \frac{1}{2}$.
It then follows that
\[b^{(r)}_n \sim \sqrt{\frac{(2^{r+1}-1)^2 - 1}{(2^{r+1}-1)^2}} \cdot \rho^{-n} \cdot \frac{n^{\alpha-1}}{\Gamma(\alpha)}\]
which simplifies to the desired formula since $\Gamma(\alpha) = \sqrt{\pi}$.

Now for $e^{(r)}_n$ we have the generating function
\[G(z; \bE^{(r)}) = \frac{1 - z - \sqrt{\big(1-(2^{r+1}-1)^2z\big)\left(1-z\right)}}{2^{r+1}(2^r-1)z}\]
by Theorem~\ref{thm:GF} which also has $\rho = \frac{1}{(2^{r+1}-1)^2}$ as its singularity closest to the origin.
As $z \to \rho$ we have that the generating function approaches
\[\frac{(2^{r+1}-1)^2}{2^{r+1}(2^r-1)} - \frac{1}{2^{r+1}(2^r-1)} - \frac{(2^{r+1}-1)^2}{2^{r+1}(2^r-1)} \cdot \sqrt{\frac{(2^{r+1}-1)^2 - 1}{(2^{r+1}-1)^2}} \cdot (1 - (2^{r+1}-1)^2x)^{-\alpha}\]
as a complex analytic function with $\alpha = -\frac{1}{2}$.
It then follows that
\[e^{(r)}_n \sim -\frac{(2^{r+1}-1)^2}{2^{r+1}(2^r-1)} \cdot \sqrt{\frac{(2^{r+1}-1)^2 - 1}{(2^{r+1}-1)^2}} \cdot \rho^{-n} \cdot \frac{n^{\alpha-1}}{\Gamma(\alpha)}\]
which simplifies to the desired formula since $\Gamma(\alpha) = - 2 \sqrt{\pi}$.
\end{proof}

For comparison to the expressions in Corollary~\ref{cor:asym} one can see that
\begin{align*}
    a^{(r)}_n &\sim (2^{2r+2})^n \cdot \frac{1}{\sqrt{\pi n}} & d^{(r)}_n &\sim (2^{2r+2})^n \cdot \frac{1}{\sqrt{\pi n^3}}
\end{align*}
which can be gotten from either the formula for coefficients in Equations~(\ref{eq:a}) and (\ref{eq:d}) or the generating functions in Equations~(\ref{eq:genA}) and (\ref{eq:genD}).

\section{Conclusion}
\label{sec:conclusion}
We now conclude with some possible directions for future work and discussion of related work.

\subsection{Intersections of hyperplanes}
Let us look at a natural generalization of the languages $\bA^{(r)}$ and $\bD^{(r)}$.
For $0 \leq j \leq r$ We can consider the language $\bA^{(r)}(j)$ of walks in $\Z^{r+1}$ with steps from $\calS_r$ which end with $x_{r-j+1} = x_{r-j+2} = \cdots = x_{r + 1} = 0$.
So, we have that $\bA^{(r)} = \bA^{(r)}(0)$.
We also take the language $\bD^{(r)}(j)$ which is contained in $\bA^{(r)}$ with the additional condition that $x_{r-j+1}, x_{r-j+1}, \dots, x_{r + 1} \geq 0$ at all times during the walk.
Similarly it is the case that $\bD^{(r)} = \bD^{(r)}(0)$.
In Theorem~\ref{thm:CFL} we saw that both $\bA^{(r)}$ and $\bD^{(r)}$ are an unambiguous CFLs which in turn guaranteed the sequences under consideration earlier were holonomic and their generating functions were algebraic.
It turns out neither $\bA^{(r)}(j)$ nor $\bD^{(r)}(j)$ is a CFL for $j > 0$.

\begin{proposition}
For any $0 < j \leq r$ the languages $\bA^{(r)}(j)$ and $\bD^{(r)}(j)$ are not a CFLs.
\end{proposition}
\begin{proof}
It will suffice to show that neither $\bA^{(1)}(1)$ nor $\bD^{(1)}(1)$ is a CFL.
We will use the pumping lemma and the choice of string pumped will readily generalize to $\bA^{(r)}(j)$ and $\bD^{(r)}(j)$ for $r, j > 1$.
Assume that either $\bA^{(1)}(1)$ or $\bD^{(1)}(1)$  is a CFL with pumping length $p$.
There is the walk
\[[(1,1)^p, (1,-1)^p, (1,1)^p, (-1,1)^p, (-1,-1)^{2p}]\ \in \bD^{(1)}(1) \subseteq \bA^{(1)}(1)\]
which cannot be pumped.
Indeed any consecutive substring of length at most $p$ will contain one coordinate with all entries equal.
Hence, after pumping this substring the walk will not end at the origin.
For other values of $r$ and $j$ the above example can be used to choose the coordinates in positions $r$ and $r+1$ while the remaining coordinates can be filled as needed to make a valid walk.
\end{proof}

Letting $a^{(r)}_n(j)$ denote the number of walks of length $2n$ in $\bA^{(r)}(j)$ we can see that
\[
    a^{(r)}_n(j) = 2^{2n(r-j)} \binom{2n}{n}^{j+1}
\]
which satisfies that recurrence
\[
    n^{j+1} a^{(r)}_n(j) = 2^{2r-j+1}(2n-1)^{j+1}a^{(r)}_{n-1}(j)
\]
so the sequence turns out to still be holomonic.
A similar computation can be performed for $\bD^{(r)}(j)$ with Catalan numbers in place of the central binomial coefficients.
We can then look at avoiding patterns like backtracking or consecutive steps.
For these languages we do not have a guarantee that the corresponding sequences are holomonic.

\begin{question}
What can be said about enumeration for the analogous languages ending on an intersection of hyperplanes?
\end{question}

\subsection{A Motzkin generalization and more patterns}
A natural extension would be to look at other sets of steps.
We outline one possibility related to Motzkin numbers.
The Motzkin numbers enumerate walks in $\mathbb{Z}$ with steps from $\{-1, 0, +1\}$ that end at the origin and are restricted to the nonnegative integers.
Continuing in the spirit of what we have done, one could consider walks in $\mathbb{Z}^{r+1}$ with step set $\{-1,0,+1\}^{r+1}$ that end with $x_{r+1} = 0$ along with combining half-space restrictions as well as pattern avoidance.
Bu~\cite{Bu} has used dynamic programming to attack enumeration of restricted Motzkin paths in a manner similar to the previously mentioned work on Dyck paths by Ekhad and Zeilberger~\cite{EZ}.
As previously mentioned these works look at other patterns (e.g. longer runs $[v,v,\dots, v]$ of consecutive steps).
One could look at the pattern of a longer run with either our current step set or another step set.
Theorem~\ref{thm:CFL} can be easily adapted to give an algebraic generating function for other patterns since the language of walks avoiding a pattern is regular.

\begin{question}
What can be said about enumeration for the analogous languages using the Motzkin-like alphabet $\{-1,0,+1\}^{r+1}$?
\end{question}

\subsection{Other related work}
We finish by mentioning some other related work.
The $r=1$ case of Theorem~\ref{thm:B} agrees with A082298 in the OEIS~\cite{OEIS} and provides a proof of a conjecture observed by David Scambler.
For $r=1$ Theorem~\ref{thm:E} and Theorem~\ref{thm:F} correspond to A086871 and A082298 respectively in the OEIS~\cite{OEIS}.
Furthermore, when $r=1$ Theorem~\ref{thm:E} is twice the formula in \cite[Theorem 2.2(b)]{Coker} which enumerates lattice walks from $(0,0)$ to $(n,n)$ using steps $\{(0,j), (j,0) : j \geq 0\}$ which never go above the line $x=y$.
These walks are also enumerated by Woan in~\cite{Woan} and are A059231 in the OEIS~\cite{OEIS}.

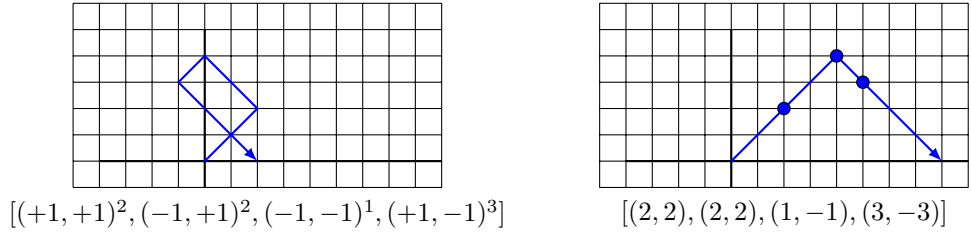
\begin{figure}
    \centering
\begin{tikzpicture}[scale=0.35]

\begin{scope}[shift={(-17.5,0)}]
\draw[step=1] (-5,-1) grid (9,6);
\draw[thick] (-4,0)--(9,0);
\draw[thick] (0,-1)--(0,5);
\draw[thick, blue, -{latex}] (0,0)--(2,2)--(0,4)--(-1,3)--(2,0);
\node (x) at (2,-2) {$[(+1,+1)^2, (-1,+1)^2, (-1,-1)^1, (+1,-1)^3]$};
\end{scope}

\begin{scope}[shift={(2.5,0)}]
\draw[step=1] (-5,-1) grid (9,6);
\draw[thick] (-4,0)--(9,0);
\draw[thick] (0,-1)--(0,5);
\draw[thick, blue, -{latex}] (0,0)--(4,4)--(8,0);
\node[draw, circle, fill=blue!, scale=0.5] (a) at (2,2) {};
\node[draw, circle, fill=blue!, scale=0.5] (b) at (4,4) {};
\node[draw, circle, fill=blue!, scale=0.5] (c) at (5,3) {};
\node (x) at (2,-2) {$[(2,2),(2,2),(1,-1),(3,-3)]$};
\end{scope}
\end{tikzpicture}
\label{fig:bijection}
\caption{A walk in $\bE'$ and the path in $\bE''$ it maps to by the bijection $\Phi$.}
\end{figure}

Let $\bE' = \bE^{(1)}((+1,+1)) \subseteq \bE^{(1)}$ be the sublanguage of walks starting with the step $(+1,+1)$.
Also, let $\bE''$ be the language of walks from $(0,0)$ to $(2n,0)$ for any $n \geq 0$ using steps from $\{(j,j), (j,-j) : j \geq 0\}$ which never go below the $x$-axis.
The walks in $\bE''$ are equinumerous with the walks enumerated by Cocker and by Woan via exchanging $(j,0) \leftrightarrow (j,j)$ and $(0,j) \leftrightarrow (j,-j)$.
Since our walks have a symmetry with orbit size $2$ for all nonempty walks by the action of reflecting over the $y$-axis, there is a bijection between $\bE'$ and $\bE''$.
Let us now define a map $\Phi: \bE' \to \bE''$.
For $w\in \bE'$ start by considering the decomposition into maximal runs of consecutive steps
\[w= [v_1^{j_1}, v_2^{j_2}, \dots v_{\ell}^{j_{\ell}}] \in \bE'\]
for $v_i \in \calS_1$ with $v_i \neq v_{i+1}$ which is unique and well-defined.
We then set
\[\Phi(w) = [u_1, u_2, \dots, u_{\ell}] \in \bE''\]
where
\[u_i = \begin{cases}
(j_i, j_i) & \text{if $v_i$ has positive $y$-coordinate;}\\
(j_i, -j_i) & \text{if $v_i$ has negative $y$-coordinate;}\\
\end{cases}
\]
for each $1 \leq i \leq \ell$.

\begin{proposition}
The map $\Phi: \bE' \to \bE''$ is a bijection.
\end{proposition}
\begin{proof}
If $w \in \bE'$ has $2n$ steps, then $\Phi(w) \in \bE''$ is a path from $(0,0)$ to $(2n,0)$.
It is enough to show $\Phi$ is injective since we know the number of walks in $\bE'$ with $2n$ steps is the same as the number of paths in $\bE''$ from $(0,0)$ to $(2n,0)$.
Consider $w, w' \in \bE'$ with 
\begin{align*}
    w &= [v_1^{j_1}, v_2^{j_2}, \dots v_{\ell}^{j_{\ell}}]\\
    w' &= [u_1^{ j'_1}, u_2^{j'_2}, \dots u_{\ell'}^{ j_{\ell'}'}]
\end{align*}
as decompositions into maximal runs.
Assume the $\Phi(w) = \Phi(w')$, then immediately we have $\ell = \ell'$ because the paths $\Phi(w)$ and $\Phi(u)$ must have the same number of steps.
Next it must be the case that $v_1 = u_1 = (+1, +1)$ since $w, w' \in \bE'$.
The first steps of $\Phi(w)$ and $\Phi(w')$ are $(j_1, j_1)$ and  $(j'_1, j'_1)$ so we have $j_1 = j'_1$.
This establishes the base case for induction.
Assume that $v_{i-1} = u_{i-1}$ and $j_{i-1} = j'_{i-1}$ for some $1 < i < \ell$.
We will show this implies that $v_i = u_i$ and $j_i = j'_i$.
The $i$th steps of $\Phi(w)$ and $\Phi(w')$ are $(j_i, \pm j_i)$ and $(j'_i, \pm j'_i)$.
This means that $j_i = j'_i$ and the $y$-coordinate of $v_i$ and $u_i$ match.
Since the walks avoid backtracking and we have decomposed them into maximal runs there is only one choice for the $x$-coordinate once the $y$-coordinate is fixed.
As $u_{i-1} = v_{i-1}$ we have that $u_i = v_i$.
The proposition then follows by induction.
\end{proof}

\begin{example}
Consider $w \in \bB'$ where
\begin{align*}
    w &= [(+1,+1), (+1,+1), (-1,+1), (-1,+1), (-1,-1), (+1,-1), (+1,-1), (+1,-1)]\\
    &= [(+1,+1)^2, (-1,+1)^2, (-1,-1)^1, (+1,-1)^3]
\end{align*}
then
\[\Phi(w) = [(2,2),(2,2),(1,-1),(3,-3)]\]
and these walks are shown in Figure~\ref{fig:bijection}.
\end{example}

\bibliographystyle{alphaurl}
\bibliography{refs}

\appendix
\section*{Appendix: Proofs with Maple code}
Here in the~Appendix we show how the functions \texttt{sumrecursion} and \texttt{sumtohyper} from the \texttt{sumtools} package in Maple~\cite{maple} can be used to compute the recurrences and hypergeometric evaluations in this paper.
We first illustrate how to compute the recurrence relations in Theorem~\ref{thm:B}, Theorem~\ref{thm:C}, Theorem~\ref{thm:E}, and Theorem~\ref{thm:F} using Maple to execute to Zeilberger's algorithm~\cite{Z90,Z91}.
Running the following in Maple
\begin{lstlisting}
with(sumtools);

sumrecursion(2*(2^r - 1)^(2*k - 2)*2^(r*(2*n - 2*k + 1))*binomial(n - 1, k - 1)*((2^r - 1)*binomial(n - 1, k - 1) + 2^r*binomial(n - 1, k - 2)), k, b(n));

sumrecursion(2*(2^r - 1)^(2*n - 2*k)*2^(r*(2*k - 1))*binomial(n - 1, k - 1)*(2^r*binomial(n - 1, k - 1) + (2^r - 1)*binomial(n - 1, k - 2)), k, c(n));

sumrecursion((1/n)*(2^r-1)^(2*k-1)*2^(r*(2*n-2*k+1))*binomial(n,k)*binomial(n,k-1), k, e(n));

sumrecursion((1/n)*(2^r-1)^(2*n-2*k)*2^(2*r*k)*binomial(n,k)*binomial(n,k-1), k, f(n));
\end{lstlisting}
will return is the relevant parts of Theorem~\ref{thm:B}, Theorem~\ref{thm:C}, Theorem~\ref{thm:E}, and Theorem~\ref{thm:F} giving the recurrences.
Also running the following in Maple
\begin{lstlisting}
with(sumtools);

sumtohyper(2*(2^r - 1)^(2*k - 2)*2^(r*(2*n - 2*k + 1))*binomial(n - 1, k - 1)*((2^r - 1)*binomial(n - 1, k - 1) + 2^r*binomial(n - 1, k - 2)), k);

sumtohyper(2*(2^r - 1)^(2*n - 2*k)*2^(r*(2*k - 1))*binomial(n - 1, k - 1)*(2^r*binomial(n - 1, k - 1) + (2^r - 1)*binomial(n - 1, k - 2)), k);

sumtohyper((2^r - 1)^(2*k - 1)*2^(r*(2*n - 2*k + 1))*binomial(n, k)*binomial(n, k - 1)/n, k);

sumtohyper((2^r - 1)^(2*n - 2*k)*2^(2*r*k)*binomial(n, k)*binomial(n, k - 1)/n, k);
\end{lstlisting}
verifies the hypergeometric evaluations.

\end{document}